\font \sevenrm=cmr7
\font \fiverm=cmr5
\documentclass[11pt, english]{amsart}
\usepackage{amsfonts}
\usepackage{amssymb}
\usepackage{amsmath}
\usepackage{color}
\usepackage{graphicx}
\usepackage{xspace}
\usepackage{axodraw}

 \newcommand{\nc}{\newcommand}

 {\everymath{\displaystyle\everymath{}}\array}%
 {\endarray}

\newtheorem{thm}{Theorem}

\newtheorem{cor}[thm]{Corollary}

\newtheorem{lem}[thm]{Lemma}
\newtheorem{prop}[thm]{Proposition}

\nc{\comment}[1]{[[{\tt #1}]] }
\nc{\Cal}[1]{{\mathcal {#1}}}
\nc{\mop}[1]{\mathop{\hbox {\rm #1} }\nolimits}
\nc{\gmop}[1]{\mathop{\hbox {\bf #1} }\nolimits}

\nc{\smop}[1]{\mathop{\hbox {\sevenrm #1} }\nolimits}
\nc{\ssmop}[1]{\mathop{\hbox {\fiverm #1} }\nolimits}
\nc{\mopl}[1]{\mathop{\hbox {\rm #1} }\limits}
\nc{\smopl}[1]{\mathop{\hbox {\sevenrm #1} }\limits}
\nc{\ssmopl}[1]{\mathop{\hbox {\fiverm #1} }\limits}
\nc{\frakg}{{\frak g}}
\nc{\g}[1]{{\frak {#1}}}
\def \restr#1{\mathstrut_{\textstyle |}\raise-6pt\hbox{$\scriptstyle #1$}}
\def \srestr#1{\mathstrut_{\scriptstyle |}\hbox to
  -1.5pt{}\raise-4pt\hbox{$\scriptscriptstyle #1$}}
\nc{\wt}{\widetilde} \nc{\wh}{\widehat}

\nc{\redtext}[1]{\textcolor{red}{\tt [[#1]]}}
\nc{\bluetext}[1]{\textcolor{blue}{#1}}

\nc\fleche[1]{\mathop{\hbox to #1 mm{\rightarrowfill}}\limits}
\nc{\ignore}[1]{}
\def\semi{\mathrel{\times}\kern -.85pt\joinrel\mathrel{\raise
    1.4pt\hbox{${\scriptscriptstyle |}$}}}
\nc\R{{\mathbb R}}
\nc\N{{\mathbb N}}
\nc\inver{^{-1}}
\nc\point{\hbox{\bf .}}
\nc\un{\hbox{\bf 1}}
\def\link#1#2{\raise -2pt\hbox{$\scriptstyle #1-\!\!-\!\!- #2$}}
\def\slink#1#2{\raise -1.5pt\hbox{$\scriptscriptstyle #1-\!\!\!-\!\!\!- #2$}}

\def\racine{{\scalebox{0.3}{ 
\begin{picture}(12,12)(38,-38)
\SetWidth{0.5} \SetColor{Black} \Vertex(45,-33){5.66}
\end{picture}}}}

\def\racinebis{{\scalebox{0.3}{ 
\begin{picture}(12,12)(38,-38)
\SetWidth{1} \SetColor{Black} \Vertex(45,-33){5.66}
\SetColor{White}
\Vertex(45,-33){4.66}
\end{picture}}}}
  
 \def\arbrea{\,{\scalebox{0.15}{ 
  \begin{picture}(8,55) (370,-248)
    \SetWidth{2}
    \SetColor{Black}
    \Line(374,-244)(374,-200)
    \Vertex(374,-197){9}
    \Vertex(375,-245){12}
  \end{picture}
}}\,}

 \def\arbreba{\,{\scalebox{0.15}{ 
\begin{picture}(8,106) (370,-197)
    \SetWidth{2}
    \SetColor{Black}
    \Line(374,-193)(374,-149)
    \Vertex(374,-146){9}
    \Vertex(375,-194){12}
    \Line(374,-142)(374,-98)
    \Vertex(374,-95){9}
  \end{picture}
}}\,}

 \def\arbrebb{\,{\scalebox{0.15}{ 
  \begin{picture}(48,48) (349,-255)
    \SetWidth{2}
    \SetColor{Black}
    \Vertex(375,-252){12}
    \Line(376,-250)(395,-215)
    \Line(373,-251)(354,-214)
    \Vertex(353,-211){9}
    \Vertex(395,-213){9}
  \end{picture}
}}}

\def\arbreca{\,{\scalebox{0.15}{
\begin{picture}(8,156) (370,-147)
    \SetWidth{2}
    \SetColor{Black}
    \Line(374,-143)(374,-99)
    \Vertex(374,-96){9}
    \Vertex(375,-144){12}
    \Line(374,-92)(374,-48)
    \Vertex(374,-45){9}
    \Line(374,-42)(374,2)
    \Vertex(374,5){9}
  \end{picture}
}}\,}

\def\arbrecabis{\,{\scalebox{0.15}{
\begin{picture}(8,156) (370,-147)
    \SetWidth{2}
    \SetColor{Black}
    \Line(374,-143)(374,-99)
    \Vertex(374,-96){9}
    \Vertex(375,-144){12}
    \Line(374,-92)(374,-48)
    \Vertex(374,-45){9}
    \Line(374,-42)(374,2)
    \Vertex(374,5){9}
    \SetColor{White}
    \Vertex(375,-144){10}
 \end{picture}
}}\,}

\def\arbrecb{\,{\scalebox{0.15}{
\begin{picture}(48,94) (349,-255)
\SetWidth{2}
\SetColor{Black}
\Line(376,-204)(395,-169)
\Line(373,-205)(354,-168)
\Vertex(353,-165){9}
\Vertex(395,-167){9}
\Vertex(374,-205){9}
\Line(374,-246)(374,-209)
\Vertex(374,-252){12}
\end{picture}}}\,}

\def\arbrecbbis{\,{\scalebox{0.15}{
\begin{picture}(48,94) (349,-255)
\SetWidth{2}
\SetColor{Black}
\Line(376,-204)(395,-169)
\Line(373,-205)(354,-168)
\Vertex(353,-165){9}
\Vertex(395,-167){9}
\Vertex(374,-205){9}
\Line(374,-246)(374,-209)
\Vertex(374,-252){12}
\SetColor{White}
\Vertex(353,-165){7}
\Vertex(395,-167){7}
\Vertex(374,-205){7}
\end{picture}}}\,}

\def\arbrecc{\,{\scalebox{0.15}{
 \begin{picture}(48,98) (349,-205)
    \SetWidth{2}
    \SetColor{Black}
    \Vertex(375,-202){12}
    \Line(376,-200)(395,-165)
    \Line(373,-201)(354,-164)
    \Vertex(353,-161){9}
    \Vertex(395,-163){9}
    \Line(353,-160)(353,-113)
    \Vertex(353,-111){9}
  \end{picture}
}}\,}

\def\arbrecd{\,{\scalebox{0.15}{
\begin{picture}(48,52) (349,-251)
    \SetWidth{2}
    \SetColor{Black}
    \Vertex(375,-248){12}
    \Line(376,-246)(395,-211)
    \Line(373,-247)(354,-210)
    \Vertex(353,-207){9}
    \Vertex(395,-209){9}
    \Line(375,-247)(375,-206)
    \Vertex(376,-203){9}
  \end{picture}
 }}\,}

\def\arbreda{\,{\scalebox{0.15}{
\begin{picture}(8,204) (370,-99)
    \SetWidth{2}
    \SetColor{Black}
    \Line(374,-95)(374,-51)
    \Vertex(374,-48){9}
    \Vertex(375,-96){12}
    \Line(374,-44)(374,0)
    \Vertex(374,3){9}
    \Line(374,6)(374,50)
    \Vertex(374,53){9}
    \Line(374,53)(374,98)
    \Vertex(374,101){9}
  \end{picture}
}}\,}

\def\arbredb{\,{\scalebox{0.15}{
\begin{picture}(48,135) (349,-255)
    \SetWidth{2}
    \SetColor{Black}
    \Line(376,-163)(395,-128)
    \Line(373,-164)(354,-127)
    \Vertex(353,-124){9}
    \Vertex(395,-126){9}
    \Vertex(374,-164){9}
    \Line(374,-205)(374,-168)
    \Vertex(374,-207){9}
    \Line(374,-248)(374,-211)
    \Vertex(374,-252){12}
  \end{picture}
}}\,}

\def\arbredc{\,{\scalebox{0.15}{
 \begin{picture}(48,150) (349,-205)
    \SetWidth{2}
    \SetColor{Black}
    \Line(376,-148)(395,-113)
    \Line(373,-149)(354,-112)
    \Vertex(353,-109){9}
    \Vertex(395,-111){9}
    \Line(353,-108)(353,-61)
    \Vertex(353,-59){9}
    \Line(374,-200)(374,-153)
    \Vertex(374,-149){9}
    \Vertex(374,-202){12}
  \end{picture}
}}\,}

\def\arbredd{\,{\scalebox{0.15}{
 \begin{picture}(48,99) (349,-251)
    \SetWidth{2}
    \SetColor{Black}
    \Line(376,-199)(395,-164)
    \Line(373,-200)(354,-163)
    \Vertex(353,-160){9}
    \Vertex(395,-162){9}
    \Vertex(376,-156){9}
    \Vertex(376,-248){12}
    \Line(375,-245)(375,-204)
    \Line(375,-200)(375,-159)
    \Vertex(375,-201){9}
  \end{picture}
}}\,}

\def\arbrede{\,{\scalebox{0.15}{
 \begin{picture}(48,153) (349,-150)
    \SetWidth{2}
    \SetColor{Black}
    \Vertex(375,-147){12}
    \Line(376,-145)(395,-110)
    \Line(373,-146)(354,-109)
    \Vertex(353,-106){9}
    \Vertex(395,-108){9}
    \Line(353,-105)(353,-58)
    \Vertex(353,-56){9}
    \Line(353,-52)(353,-5)
    \Vertex(353,-1){9}
  \end{picture}
}}\,}

\def\arbredf{\,{\scalebox{0.15}{
\begin{picture}(48,98) (349,-205)
    \SetWidth{2}
    \SetColor{Black}
    \Vertex(375,-202){12}
    \Line(376,-200)(395,-165)
    \Line(373,-201)(354,-164)
    \Vertex(353,-161){9}
    \Vertex(395,-163){9}
    \Line(353,-160)(353,-113)
    \Vertex(353,-111){9}
    \Line(395,-159)(395,-112)
    \Vertex(395,-111){9}
  \end{picture}
}}\,}

\def\arbredz{\,{\scalebox{0.15}{
  \begin{picture}(68,88) (329,-215)
    \SetWidth{2}
    \SetColor{Black}
    \Vertex(375,-212){12}
    \Line(376,-210)(395,-175)
    \Line(373,-211)(354,-174)
    \Vertex(353,-171){9}
    \Vertex(395,-173){9}
    \Line(351,-168)(332,-131)
    \Line(355,-168)(374,-133)
    \Vertex(333,-131){9}
    \Vertex(374,-131){9}
  \end{picture}
}}\,}

\def\arbredg{\,{\scalebox{0.15}{
\begin{picture}(48,98) (349,-205)
    \SetWidth{2}
    \SetColor{Black}
    \Vertex(375,-202){12}
    \Line(376,-200)(395,-165)
    \Line(373,-201)(354,-164)
    \Vertex(353,-161){9}
    \Vertex(395,-163){9}
    \Line(375,-201)(375,-160)
    \Vertex(376,-157){9}
    \Vertex(376,-111){9}
    \Line(375,-155)(375,-114)
  \end{picture}
}}\,}

\def\arbredh{\,{\scalebox{0.15}{
 \begin{picture}(90,46) (330,-257)
    \SetWidth{2}
    \SetColor{Black}
    \Vertex(375,-254){12}
    \Line(376,-252)(395,-217)
    \Vertex(395,-215){9}
    \Line(374,-254)(335,-226)
    \Vertex(334,-224){9}
    \Line(375,-252)(356,-215)
    \Vertex(355,-215){9}
    \Line(374,-255)(417,-227)
    \Vertex(418,-225){9}
  \end{picture}
}}\,}

 \def\arbreeb{\,{\scalebox{0.15}{
 \begin{picture}(48,153) (349,-150)
    \SetWidth{2}
    \SetColor{Black}
    \Vertex(375,-147){12}
    \Line(375,-147)(395,-108)
    \Line(375,-147)(354,-108)
    \Vertex(353,-108){9}
    \Vertex(395,-108){9}
    \Vertex(395,-56){9}
    \Line(353,-108)(353,-56)
    \Line(395,-108)(395,-56)
    \Vertex(353,-56){9}
    \Line(353,-56)(353,-1)
    \Vertex(353,-1){9}
  \end{picture}
}}\,}

  \def\arbreec{\,{\scalebox{0.15}{
 \begin{picture}(48,153) (349,-150)
    \SetWidth{2}
    \SetColor{Black}
    \Vertex(375,-147){12}
    \Line(375,-147)(395,-108)
    \Line(375,-147)(354,-108)
    \Vertex(353,-108){9}
    \Vertex(395,-108){9}
    \Line(353,-108)(353,-56)
    \Vertex(353,-56){9}
    \Line(353,-56)(353,-10)
    \Vertex(353,-10){9}
    \Line(353,-10)(353,35)
    \Vertex(353,35){9}
  \end{picture}
}}\,}

\def\arbreed{\,{\scalebox{0.15}{
\begin{picture}(8,204) (370,-99)
    \SetWidth{2}
    \SetColor{Black}
    \Line(374,-95)(374,-51)
    \Vertex(374,-48){9}
    \Vertex(375,-96){12}
    \Line(374,-44)(374,0)
    \Vertex(374,3){9}
    \Line(374,6)(374,50)
    \Vertex(374,53){9}
    \Line(374,53)(374,95)
    \Vertex(374,95){9}
     \Line(374,95)(374,130)
    \Vertex(374,130){9}
  \end{picture}
}}\,}

\def\arbreee{\,{\scalebox{0.15}{
 \begin{picture}(48,99) (379,-251)
    \SetWidth{2}
    \SetColor{Black}
    \Line(376,-201)(395,-162)
    \Line(376,-201)(354,-162)
    \Vertex(353,-162){9}
    \Vertex(399,-162){9}
    \Vertex(400,-201){9}
    \Vertex(400,-248){12}
    \Line(400,-248)(376,-201)
    \Line(400,-248)(424,-201)
    \Line(400,-248)(400,-201)
    \Vertex(376,-201){9}
    \Vertex(424,-201){9}
  \end{picture}
}}\,}

\def\arbreef{\,{\scalebox{0.15}{
 \begin{picture}(48,150) (349,-205)
    \SetWidth{2}
    \SetColor{Black}
    \Line(374,-149)(395,-111)
    \Line(374,-149)(353,-111)
    \Vertex(353,-111){9}
    \Vertex(395,-111){9}
    \Line(353,-109)(334,-65)
    \Line(353,-109)(372,-65)
    \Vertex(334,-65){9}
    \Vertex(372,-65){9}
    \Line(374,-202)(374,-149)
    \Vertex(374,-149){9}
    \Vertex(374,-202){12}
  \end{picture}
}}\,}

\begin{document}

\title[Rooted trees, non-rooted trees and hamiltonian B-series]{Rooted trees, non-rooted trees and hamiltonian B-series}

\author{Geir bogfjellmo}
\address{NTNU}
	   \email{bogfjell@math.ntnu.no}
 	  \urladdr{}

\author{Charles Curry}
\address{Heriot Watt University}
         \email{chc5@hw.ac.uk}         
	  \urladdr{}

\author{Dominique Manchon}
\address{Universit\'e Blaise Pascal,
         C.N.R.S.-UMR 6620, BP 80026,
         63171 Aubi\`ere, France}       
         \email{manchon@math.univ-bpclermont.fr}
         \urladdr{http://math.univ-bpclermont.fr/~manchon/}

\date{January 17th 2013}

\begin{abstract}
We explore the relationship between (non-planar) rooted trees and free trees, i.e. without root. We give in particular, for non-rooted trees, a substitute for the Lie bracket given by the antisymmetrization of the pre-Lie product.

\bigskip

\noindent {\bf{Keywords}}: Rooted trees; B-series; Trees; Hamiltonian vector fields.

\smallskip

\noindent {\bf{Math. subject classification}}: 16W30; 05C05; 16W25; 17D25; 37C10 
\end{abstract}

%
%

\maketitle





\section{Introduction}
\label{sect:intro}
A striking link between rooted trees and vector fields on an affine space $\R^n$ has been established by A. Cayley \cite{Cay} as early as 1857. The interest for this correspondence has been renewed since J. Butcher showed the key role of rooted trees for understanding Runge-Kutta methods in numerical approximation \cite{Butcher1, Br, HWL}. The modern approach to this correspondence can be summarized as follows: the product on vector fields on $\R^n$ defined by:
\begin{equation}
\left(\sum_{i=1}^n f_i\partial_i\right)\rhd\left(\sum_{i=1}^n g_j\partial_j\right):=\sum_{j=1}^n\left(\sum_{i=1}^n f_i(\partial_ig_j)\right)\partial_j
\end{equation}
is left pre-Lie, which means that for any vector fields $a,b,c$ the associator $a\rhd(b\rhd c)-(a\rhd b)\rhd c$ is symmetric with respect to $a$ and $b$. On the other hand, the free pre-Lie algebra with one generator (on some base field $k$) is the vector space $\Cal T$ spanned by the planar rooted trees \cite{ChaLiv, DL}. The generator is the one-vertex tree $\racine$, and the pre-Lie product on rooted trees is given by grafting:
\begin{equation}
s\to t=\sum_{v\in \Cal V(t)}s\to_v t,
\end{equation}
where $s\to_v t$ is the rooted tree obtained by grafting the rooted tree $s$ on the vertex $v$ of the tree $t$. Hence for any vector field $a$ on $\R^n$ there exists a unique pre-Lie algebra morphism $\Cal F_a$ from $\Cal T$ to vector fields such that $\Cal F_a(\racine)=a$. This can be generalized to an arbitrary number of generators, since the free pre-Lie algebra on a set $D$ of generators is the span of rooted trees with vertices coloured by $D$. In this case, for any collection $\underline a=(a_d)_{d\in D}$ of vector fields, there exists a unique pre-Lie algebra morphism $\Cal F_{\underline a}$ from the linear span $\Cal T_D$ of coloured trees to vector fields on $\R^n$, such that $\Cal F_{\underline a}(\racine_d)=a_d$ for any $d\in D$.\\

The vector fields $\Cal F_a(t)$ (or $\Cal F_{\underline a}(t)$ in the coloured case) are the \textsl{elementary differentials}, building blocks of the B-series \cite{HWL} which are defined as follows: for any linear form $\alpha$ on $\Cal T_D\oplus\R\un$ where $\un$ is the empty tree, for any collection of vector fields $\underline a$ and for any initial point $y_0\in \R^n$,  the corresponding B-series\footnote{Such coloured B-series are sometimes called NB-series in the literature.} is a formal series in the indeterminate $h$ given by:
\begin{equation}
B_{\underline a}(\alpha,y_0)=\alpha(\un)y_0+\sum_{t\in\Cal T_D}h^{|t|}\frac{\alpha(t)}{\mop{sym}(t)}\Cal F_{\underline a}(t)(y_0).
\end{equation}
Here $|t|$ is the number of vertices of $t$, and $\mop{sym}(t)$ is its symmetry factor, i.e. the cardinal of its automorphism group $\mop{Aut} t$. For any vector field $a$, the exact solution of the differential equation:
\begin{equation}
\dot y(t)=a\big(y(t)\big)
\end{equation}
with initial condition $y(0)=y_0$ admits a (one-coloured) B-series expansion at time $t=h$, and its approximation by any Runge-Kutta method as well \cite{Butcher1, Butcher2, HWL}. The formal transformation $y_0\mapsto B_a(\alpha,y_0)$ is a formal series with coefficients in $C^\infty(\R^n,\R^n)$.\\

We will be interested in \textsl{canonical B-series} \cite{CS}, i.e. such that the formal transformation $B_{\underline a}(\alpha,-)$ is a symplectomorphism for any collection of hamiltonian vector fields $\underline a$. Here, the dimension $n=2r$ is even, and $\R^{2r}$ is endowed with the standard symplectic structure:
\begin{equation}
\omega(x,y)=\sum_{i=1}^r x_iy_{r+i}-x_{r+i}y_i,
\end{equation}
and a vector field $a=\sum_{i=1}^{2r}a_i\partial_i$ is hamiltonian if there exists a smooth map $H:\R^{2r}\to\R$ such that:
\begin{eqnarray*}
a_i&=&-\frac{\partial H}{\partial t_{i+r}}\hbox{ for }i=1,\ldots ,r, \hbox{ and}\\
a_i&=&\frac{\partial H}{\partial t_{i-r}}\hbox{ for }i=r+1,\ldots ,2r.
\end{eqnarray*}
Recall that the Poisson bracket of two smooth maps $f,g$ on $\R^{2r}$ is given by:
\begin{equation}\label{poisson-bracket}
\{f,g\}=\sum_{i=i}^r\frac{\partial f}{\partial t_i}\frac{\partial g}{\partial t_{i+r}}-\frac{\partial g}{\partial t_i}\frac{\partial f}{\partial t_{i+r}}.
\end{equation}
Hence hamiltonian vector fields are those vector fields $a$ which can be expressed as:
$$a=\{H,-\}$$
for some $H\in C^\infty(\R^{2r})$. A B-series turns out to be canonical if and only if the following condition holds for any rooted trees $s$ and $t$ \cite[Theorem 2]{AMS}:
\begin{equation}
\alpha(s\circ t)+\alpha(t\circ s)=\alpha(s)\alpha(t),
\end{equation}
where $s\circ t$ is the right Butcher product, defined by grafting the tree $t$ on the root of the tree $s$. This result is also valid in the coloured case. The infinitesimal counterpart of this result expresses as follows (\cite{HWL}, Theorem IX.9.10 for one-colour case): a B-series $B_{\underline a}(\alpha,-)$ with $\alpha(\un)=0$ defines a hamiltonian vector field for any hamiltonian vector field $a$ if and only if:
\begin{equation}\label{hbs}
\alpha(s\circ t)+\alpha(t\circ s)=0.
\end{equation}
Let us call the B-series of the type described above \textsl{hamiltonian B-series}. Our interest in non-rooted trees comes from the following elementary observation: the two rooted trees $s\circ t$ and $t\circ s$ are equal as non-rooted trees, and one is obtained from the other by shifting the root to a neighbouring vertex. As an easy consequence of \eqref{hbs}, any hamiltonian B-series $B_{\underline a}(\alpha,-)$ has to satisfy that if two rooted trees $s$ and $t$ are equal as non-rooted trees, then:
\begin{equation}
\alpha(s)=\pm\alpha(t).
\end{equation}
This implies that, modulo a careful account of the signs involved, hamiltonian B-series are naturally indexed by non-rooted trees rather than by rooted ones. The sign is plus or minus according to the parity of the minimal number of ''root shifts" $s_1\circ s_2\mapsto s_2\circ s_1$ that are required to change $s$ into $t$.\\

In the present paper we address the following question: \textsl{what survives from the pre-Lie structure at the level of non-rooted trees?} There is a natural linear map $\wt X$ from non-rooted trees to (the linear span of) rooted trees, sending a tree to the sum of all its rooted representatives, with alternating signs. Its precise definition involves a total order on rooted trees introduced by A. Murua \cite{M}. We propose a binary product $\diamond$ on the linear span of non-rooted trees, which is roughly speaking an alternating sum of all trees obtained by linking a vertex of the first tree with a vertex of the second tree. Theorem \ref{main} is the key result of the paper. It implies the fact that $\diamond$ is a Lie bracket and that $\wt X$ is a Lie algebra morphism, the Lie bracket on rooted trees being given by antisymmetrizing the pre-Lie product.\\

\noindent
\textbf{Acknowledgements :} This article came out from a workshop in December 2012 at NTNU in Trondheim. The authors thank  Elena Celledoni, Kurusch Ebrahimi-Fard, Brynjulf Owren and all the participants for illuminating discussions. The third author also thanks Ander Murua and Jesus Sanz-Serna for sharing references and for their encouragements. This work is partly supported by Campus France, PHC Aurora 24678ZC. The third author also acknowledges a support by Agence Nationale de la Recherche (projet CARMA).
\section{Structural facts about non-rooted trees}
We denote by $T$ (resp. $FT$) the set of non-planar rooted (resp. non-rooted) trees. We denote by $\Cal T$ (resp. $\Cal {FT}$) the vector spaces freely generated by $T$ (resp. $FT$). The projection $\pi:T\to\!\!\!\!\!\to FT$ is defined by forgetting the root. It extends linearly to $\pi:\Cal T \to\!\!\!\!\!\to\Cal {FT}$. Rooted trees will be denoted by latin letters $s,t,\ldots$, non-rooted trees by greek letters $\sigma,\tau,\ldots$. We will also use "free tree" as a synonymous for "non-rooted tree". For any free tree $\tau$ and for any vertex $v$ of $\tau$, we denote by $\tau_v$ the unique rooted tree built from $\tau$ by putting the root at $v$.
\subsection{A total order on rooted trees}
Recall that any rooted tree $t$ is obtained by grafting rooted trees $t_1,\ldots,t_q$ on a common root:
$$t=B_+(t_1,\ldots,t_q).$$
The trees $t_j$ are called the \textsl{branches} of $t$. A. Murua defines in \cite{M} a total order on the set of (one-colour) rooted trees in a recursive way as follows: the \textsl{canonical decomposition} of a tree $t$ is given by $t=t_L\circ t_R$ where $t_R$ is the maximal branch of $t$. The maximality is to be understood with respect to the total order, supposed to be already defined for trees with number of vertices strictly smaller than $|t|$. Then $s<t$ if and only if:
\begin{itemize}
\item either $|s|<|t|$,
\item or $|s|=|t|$ and $s_L<t_L$,
\item or $|s|=|t|$, $s_L=t_L$ and $s_R<t_R$.
\end{itemize}
In the one-colour case, the total order of the first few trees is:
$$\racine<\arbrea<\arbreba<\arbrebb<\arbreca<\arbrecb<\arbrecc<\arbrecd<\arbreda<\arbredb<\arbredc<\arbredd<\arbrede<\arbredz<\arbredf<\arbredg<\arbredh<\cdots$$
If we prescribe a total order on the set of colours $D$ and allow the set of one node coloured trees to inherit this order, incorporating this into the definition above gives a total order on the set of coloured rooted trees. Note that the structure of the one-colour order is not entirely preserved, as, for example, for two colours $\racine <\racinebis$ we have $\arbrecabis >\arbrecbbis$ whereas $\arbreca < \arbrecb$.\\

\subsection{Superfluous trees}
This notion has been introduced in \cite{AS}, where the authors describe order conditions for canonical B-series coming from Runge-Kutta approximation methods. Let $B_{\underline a}(\alpha,-)$ be a hamiltonian B-series. According to \eqref{hbs}, we have $\alpha(t\circ t)=0$ for any rooted tree $t$. Any non-rooted tree $\tau$ such that there exists a rooted tree $s$ with $s\circ s\in\pi^{-1}(\tau)$ is called a \textsl{superfluous tree}, and a rooted tree $t$ is said to be superfluous if its underlying free tree $\pi(t)$ is. Such trees never appear in a hamiltonian B-series. For any free tree $\tau\in FT$, its \textsl{canonical representative} is the maximal element of the set $\pi^{-1}(\tau)\subset T$ for the total order above. The following lemma gives a characterization of superfluous trees:
\begin{lem}\label{lem:superfluous}
Let $\tau\in FT$ have two distinct vertices $v$ and $w$ such that $\tau_v=\tau_w$ is the canonical representative of $\tau$. Then:
\begin{enumerate}
\item $v$ and $w$ are the two ends of a common edge in $\tau$,
\item There exists $s\in T$ such that $\tau_v=\tau_w=s\circ s$.
\end{enumerate} 
\end{lem}
\begin{proof}
First of all, the maximal branch of $\tau_v$ contains $w$ (and vice-versa). Indeed, Suppose the maximal branch of $\tau_v$ does not contain $w$ (and hence vice-versa). Let
$$ \tau_v = B^+(t_1,t_2,\ldots,t_n,t_w,t_{\mathrm{max}}),
\hskip 10mm \tau_w = B^+(t'_1,t'_2,\ldots,t'_n,t'_v,t'_{\mathrm{max}}),$$
where $t_w$ is the branch of $\tau_v$ containing $w$ and $t'_v$ similarly. It is clear that $t'_v$ contains all branches of $\tau_v$ except $t_w$. Hence $|t'_v|>|t_1|+\ldots +|t_n|+|t_{\mathrm{max}}|$ and as $|t_{\mathrm{max}}|=|t'_{\mathrm{max}}|$
we have $|t'_v|>|t'_{\mathrm{max}}|$, a contradiction. Now suppose that $v$ and $w$ are not neighbours, and choose a vertex $x$ between $v$ and $w$, i.e. such that there is a path from $v$ to $w$ of meeting $x$. The maximal branch of $\tau_x$ cannot contain both $v$ and $w$; suppose it does not contain $v$. Then it is a subtree of the maximal branch $\tau_v$ and hence contains strictly less vertices. Looking at the canonical decompositions:
$$t:=\tau_v=\tau_w=t_L\circ t_R,\hskip 10mm t':=\tau_x=t'_L\circ t'_R,$$
we have then $|t'_L|>|t_L|$, which immediately yields $\tau_x>\tau_v$, which is a contradiction. This proves the first assertion, and the second assertion follows immediately.
\end{proof}

\noindent
There are four superfluous free trees with six vertices or less.  The corresponding superfluous rooted trees are:
$$\arbrea,\hskip 12mm \arbreca, \arbrecc, \hskip 12mm \arbreed, \arbreec, \arbreeb,\hskip 12mm  \arbreee, \arbreef.$$
We denote by $S$ the set of superfluous free trees and by $FT'$ the set of non-superfluous trees, hence $FT=FT'\amalg S$. The corresponding linear spans will be denoted by $\Cal S$ and $\Cal{FT}'$. We have $\Cal{FT}=\Cal S\oplus\Cal {FT}'$, which leads to a linear isomorphism:
$$\Cal {FT}'\sim\Cal{FT}/\Cal S.$$
\subsection{Symmetries}
We keep the notations of the previous subsection. For any non-superfluous tree $\tau\in FT'$ we denote by $*$ the unique vertex such that $\tau_*$ is the canonical representative of $\tau$. The group of automorphisms of $\tau$ is the subgroup $\mop{Aut}\tau$ of the group of permutations $\varphi$ of $\Cal V(\tau)$ which respect the tree structure, i.e. such that, for any $v,w\in\Cal V(\tau)$, there is an edge between $v$ and $w$ if and only if there is an edge between $\varphi(v)$ and $\varphi(w)$.\\

For any rooted tree $t$ we also denote by $\mop{Aut} t$ its group of automorphisms, i.e. the subgroup of the group of permutations $\varphi$ of $\Cal V(t)$ which respect the rooted tree structure. It obviously coincides with the stabilizer of the root in $\mop{Aut} \pi(t)$. Now for any non-superfluous free tree $\tau$ it is obvious from Lemma \ref{lem:superfluous} that $\mop{Aut}\tau$ fixes the vertex $*$, hence $\mop{Aut}\tau=\mop{Aut}\tau_*$.\\

Now $\mop{Aut}\tau$ acts on the set of vertices $\Cal V(\tau)$. Moreover, for any vertex $v$ this group acts transitively on the subset of possible roots for $\tau_v$, namely:
$$\Cal R_v(\tau):=\{w\in\Cal V(\tau),\, \tau_w\sim\tau_v\}.$$
Hence $R_v(\tau)$ identifies itself with the homogeneous space:
\begin{equation}
R_v(\tau)\sim\mop{Aut}\tau_*/\mop{Aut}\tau_v.
\end{equation}
This immediately leads to the following proposition, which is implicit in the proof of Lemma IX.9.7 in \cite{HWL}:
\begin{prop}\label{prop:sym}
Let $\tau$ be a non-superfluous free tree, let $t$ be a rooted tree such that $\pi(t)=\tau$, and let $N(t,\tau)$ be the number of vertices $v\in\Cal V(\tau)$ such that $\tau_v=t$. Then:
\begin{equation}
N(t,\tau)=\frac{\mop{sym}(\tau_*)}{\mop{sym}( t)}.
\end{equation}
\end{prop}
\subsection{Grafting and linking}
Let $\sigma$ and $\tau$ be two non-rooted trees, and let us choose a vertex $v$ of $\sigma$ and a vertex $w$ of $\tau$. We will denote by $\sigma\link{v}{w}\tau$ the non-rooted tree obtained by taking $\sigma$ and $\tau$ together and adding a new edge between $v$ and $w$. This linking operation is related to grafting of rooted trees as follows: for any other choice of vertices $x$ of $\sigma$ and $y$ of $\tau$ we have:
\begin{eqnarray}
(\sigma\link{v}{w}\tau)_y&=&\sigma_v\to_w\tau_y,\label{gl1}\\
(\sigma\link{v}{w}\tau)_x&=&\tau_w\to_v\sigma_x.\label{gl2}
\end{eqnarray}
\section{A binary operation on non-rooted trees}
The linear map $\wt X:\Cal{FT}\to\Cal T$ is defined for any non-rooted tree $\tau$ by:
\begin{equation}
\wt X(\tau)=\sum_{v\in\Cal V(\tau)}\varepsilon(v,\tau)\tau_v \label{kappa},
\end{equation}
and extended linearly. Here $\varepsilon(v,\tau)$ is equal to $0$ if $\tau$ is superfluous, and is equal to $1$ (resp. $-1$) if $\tau$ is not superfluous and if the number of requested root shifts to change $\tau_v$ into the canonical representative of $\tau$ is even (resp. odd). This number, which we denote by $\kappa(v,\tau)$, is indeed unambiguous for non-superfluous trees according to Lemma \ref{lem:superfluous}. We obviously have:
\begin{equation}\label{invkappa}
\varepsilon(v,\tau)=\varepsilon\big(\varphi(v),\tau\big)
\end{equation}
for any $\varphi\in\mop{Aut}\tau$. The introduction of the map $\wt X$ is justified by the fact that, according to \eqref{hbs}, \eqref{invkappa} and Proposition \ref{prop:sym}, rooted trees involved in hamiltonian B-series do group themselves under terms $\wt X(\tau)$ with $\tau\in FT$. Indeed,
\begin{prop}
\begin{equation}
B_{\underline a}(\alpha,-)=\sum_{\tau\in FT}h^{|\tau|}\frac{\alpha(\tau_*)}{\mop{sym}(\tau_*)}\Cal F_{\underline a}\big(\wt X(\tau)\big).
\end{equation}
\end{prop}
\noindent
Now let us define a binary product on $\Cal {FT}$ by the formula:
\begin{equation}
\sigma\diamond\tau=\sum_{v\in\Cal V(\sigma),\,w\in\Cal V(\tau)}\delta(v,w)\sigma\link{v}{w}\tau,
\end{equation}
with $\delta(v,w):=\varepsilon(w, \sigma\link{v}{w}\tau)\varepsilon(v,\sigma)\varepsilon(w,\tau)$.
\begin{thm}\label{main}
We have $\sigma\diamond\tau\in\Cal{FT}'$ for any $\sigma,\tau\in\Cal{FT}$, and $\sigma\diamond\tau=0$ if $\sigma$ or $\tau$ is superfluous. The product $\diamond$ is antisymmetric, and the following relation holds:
\begin{equation}
\wt X(\sigma\diamond\tau)=\wt X(\sigma)\to\wt X(\tau)-\wt X(\tau)\to\wt X(\sigma)=[\wt X(\sigma),\, \wt X(\tau)].
\end{equation}
\end{thm}
\begin{proof}
A computation of the left-hand side gives:
\begin{eqnarray*}
\wt X(\sigma\diamond\tau)&=&\sum_{v,x\in\Cal V(\sigma),\,w\in\Cal V(\tau)}
	\varepsilon(x,\sigma\link{v}{w}\tau)\varepsilon(w,\sigma\link{v}{w}\tau)\varepsilon(v,\sigma)\varepsilon(w,\tau)(\sigma\link{v}{w}\tau)_x\\
	&+&\sum_{v\in\Cal V(\sigma),\,w,y\in\Cal V(\tau)}
	\varepsilon(y,\sigma\link{v}{w}\tau)\varepsilon(w,\sigma\link{v}{w}\tau)\varepsilon(v,\sigma)\varepsilon(w,\tau)(\sigma\link{v}{w}\tau)_y,
\end{eqnarray*}
and computing the right-hand side gives:
\begin{eqnarray*}
[\wt X(\sigma),\,\wt X(\tau)]=&-&\sum_{v,x\in\Cal V(\sigma),\,w\in\Cal V(\tau)}
	\varepsilon(v,\sigma)\varepsilon(w,\tau)\tau_w\to_x\sigma_v\\
	&+&\sum_{v\in\Cal V(\sigma),\,w,y\in\Cal V(\tau)}
	\varepsilon(v,\sigma)\varepsilon(w,\tau)\sigma_v\to_y\tau_w.
\end{eqnarray*}
Exchanging $x$ and $v$ in the first sum, and $y$ and $w$ in the second, we get:
\begin{eqnarray*}
[\wt X(\sigma),\,\wt X(\tau)]=&-&\sum_{v,x\in\Cal V(\sigma),\,w\in\Cal V(\tau)}
	\varepsilon(x,\sigma)\varepsilon(w,\tau)\tau_w\to_v\sigma_x\\
	&+&\sum_{v\in\Cal V(\sigma),\,w,y\in\Cal V(\tau)}
	\varepsilon(v,\sigma)\varepsilon(y,\tau)\sigma_v\to_w\tau_y.
\end{eqnarray*}
The first assertion is immediate since $\varepsilon(w,\,\sigma\slink vw\tau)$ vanishes if $\sigma\link vw\tau$ is superfluous. The second assertion is also immediate, since $\delta(v,w)$ vanishes if $\sigma$ or $\tau$ is superfluous. The antisymmetry comes from the fact that $v$ and $w$ are neighbours in $\sigma\link{v}{w}\tau$.\\
\begin{enumerate}
\item
If $\sigma$ or $\tau$ is superfluous, any individual term in both sides vanishes. 
\item
If $\sigma$ and $\tau$ are not superfluous it may happen that $\sigma\link vw \tau$ is superfluous for some $v\in \Cal V(\sigma)$ and $w\in\Cal V(\tau)$. The corresponding term $\wt X(\sigma\link vw\tau)$ in $\wt X(\sigma\diamond\tau)$ vanishes. On the other hand, the sum of all terms in $[\wt X(\sigma),\,\wt X(\tau)]$ corresponding to the couple $(v,w)$ chosen above writes down as:
\begin{eqnarray*}T_{v,w}:=&-&\sum_{x\in\Cal V(\sigma)}
	(-1)^{\kappa(x,\sigma)+\kappa(w,\tau)}\tau_w\to_v\sigma_x\\
	&+&\sum_{y\in\Cal V(\tau)}
	(-1)^{\kappa(v,\sigma)+\kappa(y,\tau)}\sigma_v\to_w\tau_y.
\end{eqnarray*}
The distance $d(x,v)$ between $x$ and $v$ in $\sigma$ is defined as the length of the (unique) path joining $x$ and $v$ in $\sigma$. It is clearly equal modulo $2$ to  the sum $\kappa(x,\sigma)+\kappa(v,\sigma)$. Similarly, $d(y,w)=\kappa(y,\tau)+\kappa(w,\tau)$ modulo $2$. Hence, using \eqref{gl1} and \eqref{gl2} we get:
\begin{equation*}
T_{v,w}=(-1)^{\kappa(v,\sigma)+\kappa(w,\tau)}\left(-\sum_{x\in\Cal V(\sigma)}(-1)^{d(x,v)}(\sigma\link vw\tau)_x+\sum_{y\in\Cal V(\tau)}(-1)^{d(y,w)}(\sigma\link vw\tau)_y
\right).
\end{equation*}
Now the distance $d(x,v)$ is the same if we compute it in $\sigma$ or in $\sigma\link vw\tau$, and similarly for $d(y,w)$. Finally, using the fact that $v$ and $w$ are neighbours in $\sigma\link vw\tau$, we have $d(x,w)=d(x,v)+1$ for any $x\in\Cal V(\sigma)$, the distance being computed in $\sigma\link vw\tau$. This finally gives:
\begin{equation*}
T_{v,w}=(-1)^{\kappa(v,\sigma)+\kappa(w,\tau)}\sum_{z\in\Cal V(\sigma\slink vw\tau)}(-1)^{d(z,w)}(\sigma\link vw\tau)_z,
\end{equation*}
which vanishes since $\sigma\link vw\tau$ is superfluous. 
\item
Finally, if $\sigma$, $\tau$ and $\sigma\link vw\tau$ are not superfluous, using \eqref{gl1} and \eqref{gl2}, both sides will be equal if we have:
\begin{eqnarray*}
\kappa(x,\sigma\link{v}{w}\tau)+\kappa(w,\sigma\link{v}{w}\tau)+\kappa(v,\sigma)&=&\kappa(x,\sigma)+1 \hbox{ modulo }2,\\
\kappa(y,\sigma\link{v}{w}\tau)+\kappa(w,\sigma\link{v}{w}\tau)+\kappa(w,\tau)&=&\kappa(y,\tau)\hbox{ modulo }2.
\end{eqnarray*}
Using the fact that $v$ and $w$ are neighbours, it rewrites as:
\begin{eqnarray*}
\kappa(x,\sigma\link{v}{w}\tau)+\kappa(x,\sigma)&=&\kappa(v,\sigma\link{v}{w}\tau)+\kappa(v,\sigma)\hbox{ modulo }2,\\
\kappa(y,\sigma\link{v}{w}\tau)+\kappa(y,\tau)&=&\kappa(w,\sigma\link{v}{w}\tau)+\kappa(w,\tau)\hbox{ modulo }2.
\end{eqnarray*}
These two last identities are always verified: looking for example at the right-hand side of the first one, moving vertex $v$ to a neighbour will change both $\kappa$'s by $\pm 1$. It remains then to jump from neighbour to neighbour up to $x$. The proof of the second identity is completely similar.
\end{enumerate}
\end{proof}
Using the identification of $\Cal{FT}/\Cal S$ with $\Cal{FT}'$, a straightforward consequence of Theorem \ref{main} is the following:
\begin{cor}
The linear map $\wt X$ is an injection of $\Cal{FT}'$ into $\Cal T$, and the product $\diamond: \Cal{FT}'\times \Cal{FT}'\to \Cal{FT}'$ verifies:
$$\wt X(\sigma\diamond\tau)=[\wt X(\sigma),\,\wt X(\tau)].$$
\end{cor}
As a consequence, the product $\diamond$ satisfies the Jacobi identity, and $\wt X$ is an embedding of Lie algebras.
\section{Application to elementary hamiltonians}
Keeping the previous notations, the vector field $\Cal F_{\underline a}\big(\wt X(\tau)\big)$ is hamiltonian for any (decorated) non-rooted tree $\tau$. Hence it can be uniquely written as $\{H_{\underline a}(\tau),-\}$ for some $H_{\underline a}(\tau)\in C^\infty(\R^{2r})$, called the \textsl{elementary hamiltonian} associated with $\tau$.
\begin{prop}
For any free trees $\sigma,\tau$ we have:
\begin{equation}\label{elem-ham}
\{H_{\underline a}(\sigma),\,H_{\underline a}(\tau)\}=H_{\underline a}(\sigma\diamond \tau).
\end{equation}
\end{prop}
\begin{proof}
We compute:
\begin{eqnarray*}
\big\{\{H_{\underline a}(\sigma),\,H_{\underline a}(\tau)\},-\big\}&=&\left[\{H_{\underline a}(\sigma),-\},\,\{H_{\underline a}(\tau),-\}\right]\\
&=&\left[\Cal F_{\underline a}\big(\wt X(\sigma)\big),\,\Cal F_{\underline a}\big(\wt X(\tau)\right]\\
&=&\Cal F_{\underline a}\left([\wt X(\sigma),\,\wt X(\tau)]\right)\\
&=&\Cal F_{\underline a}\circ\wt X(\sigma\diamond\tau)\\
&=&\{H_{\underline a}(\sigma\diamond\tau),-\}.
\end{eqnarray*}
One concludes by using the uniqueness of the hamiltonian representation of a hamiltonian vector field.
\end{proof}
%

\end{document}